\documentclass[a4paper,12pt]{scrartcl}
\usepackage{amsmath,amssymb,latexsym,amsthm,enumerate}
\usepackage{xcolor}
\usepackage{graphicx}
\usepackage{epstopdf}

\usepackage[T1]{fontenc}
\usepackage[utf8]{inputenc}
\usepackage[english]{babel}
\newcommand*{\wh}{\widehat}

\newcommand*{\ol}{\overline}

\newcommand*{\bbN}{\mathbb N}

\newcommand*{\bbR}{\mathbb R}

\newcommand*{\cB}{\mathcal{B}}

\newcommand*{\cF}{\mathcal{F}}

\newcommand*{\N}{\mathbb{N}}
\newcommand*{\IN}{\mathbb{N}}
\newcommand*{\IZ}{\mathbb{Z}}

\newcommand*{\IR}{\mathbb{R}}

\newcommand*{\R}{\mathbb{R}}

\newcommand*{\eps}{\varepsilon}
\newcommand*{\loc}{\mathrm{loc}}

\newcommand*{\supp}{\operatorname{supp}}

\newcommand*{\Law}{\operatorname{Law}}
\newcommand*{\Var}{\operatorname{Var}}
\newcommand*{\Cov}{\operatorname{Cov}}

\newcommand{\be}{\begin{eqnarray*}}
\newcommand{\ee}{\end{eqnarray*}}
\newcommand{\ben}{\begin{eqnarray}}
\newcommand{\een}{\end{eqnarray}}
\newcommand{\bi}{\begin{itemize}}
\newcommand{\ei}{\end{itemize}}

\newtheorem{theo}{Theorem}[section]

\newtheorem{lemma}[theo]{Lemma}
\newtheorem{propo}[theo]{Proposition}
\newtheorem{corollary}[theo]{Corollary}

\theoremstyle{definition}
\newtheorem{ex}[theo]{Example}
\newtheorem{defi}[theo]{Definition}

\newtheorem{remark}[theo]{Remark}

\title{A functional limit theorem for irregular SDEs}

\author{Stefan Ankirchner \and Thomas Kruse \and Mikhail Urusov \thanks{Stefan Ankirchner, Institute for Mathematics, University of Jena, Ernst-Abbe-Platz 2, 07745 Jena, Germany. \emph{Email:} s.ankirchner@uni-jena.de, \emph{Phone:} +49 (0)3641 946275;
Thomas Kruse, Faculty of Mathematics, University of Duisburg-Essen, Thea-Leymann-Str.~9, 45127 Essen, Germany.
\emph{Email:} thomas.kruse@uni-due.de, \emph{Phone:} +49 (0)201 183 3911;
Mikhail Urusov, Faculty of Mathematics, University of Duisburg-Essen, Thea-Leymann-Str.~9, 45127 Essen, Germany; and
Steklov Mathematical Institute,
Russian Academy of Sciences,
ul.~Gubkina~8,
119991 Moscow, Russia.
\emph{Email:} mikhail.urusov@uni-due.de, \emph{Phone:} +49 (0)201 183 7428.
Stefan Ankirchner and Thomas Kruse acknowledge the financial support from the French Banking Federation through the \emph{Chaire Markets in Transition}.
The work of Mikhail Urusov is supported by the Russian Science Foundation under grant 14-21-00162 in Steklov Mathematical Institute of Russian Academy of Sciences.
We are grateful to
Denis Belomestny,
Hans-J\"urgen Engelbert,
Martin Hutzenthaler,
Arturo Kohatsu-Higa,
Nikolaus Schweizer,
Pavel Yaskov
and seminar participants
in Le Mans, Evry, Paris,
UCL, ISFA Lyon,
Duisburg-Essen, Milan
for helpful comments.}}

\begin{document}

\maketitle

\begin{abstract}
Let $X_1, X_2, \ldots$ be a sequence of i.i.d.\ real-valued random variables with mean zero, and consider the scaled random walk of the form $Y^N_{k+1} = Y^N_{k} +  a_N(Y^N_k) X_{k+1}$, where $a_N: \R \to \R_+$. We show, under mild assumptions on the law of $X_i$, that one can choose the scale factor $a_N$ in such a way that the process $(Y^N_{\lfloor N t \rfloor})_{t \in \R_+}$ converges in distribution to a given diffusion $(M_t)_{t \in \R_+}$ solving a stochastic differential equation with possibly irregular coefficients, as $N \to \infty$. To this end we embed the scaled random walks into the diffusion $M$ with a sequence of stopping times with expected time step~$1/N$. %Moreover, we need to prove a version of the weak law of large numbers for uniformly integrable arrays.
\end{abstract}

\begin{center}\footnotesize
  \begin{tabular}{r@{ : }p{10cm}}
      {\it 2010 MSC} & 60F17, 60J60, 65C30.\\
      {\it Keywords} &  stochastic differential equations,
	irregular diffusion coefficient,
	weak law of large numbers for u.i.\ arrays,
	weak convergence of processes,
	Skorokhod embedding problem.
    \end{tabular}
  \end{center}

%==========
\section*{Introduction}
Let $X_1, X_2, \ldots$ be a sequence of i.i.d.\ integrable random variables with $E(X_i) = 0$.
%Let $S_n = X_1 + \cdots + X_n$, $n \in \N$, be the random walk generated by $(X_k)$.
Let $a_N:\R \to \R_+$ be a function depending on $N \in \N$, and let $(Y^N_k)_{k \in \IZ_+}$ be the process satisfying $Y^N_0 = m \in \R$ and
\ben\label{classic scaling}
Y^N_{k+1} = Y^N_{k} +  a_N(Y^N_k) X_{k+1}, \quad k \in \IZ_+.
\een
We extend $Y^N$ to a continuous time processes by defining $Y^N_t = Y^N_{\lfloor t \rfloor} + (t-\lfloor t \rfloor) (Y^N_{\lfloor t \rfloor + 1} - Y^N_{\lfloor t \rfloor})$.

Consider the particular case where $E(X_i^2) = 1$ and $a_N$ is constant equal to $\frac{1}{\sqrt{N}}$. Then  $(Y^N_k)$ is the random walk generated by $(X_i)$, scaled by the constant $\frac{1}{\sqrt{N}}$, and Donsker's theorem implies that the continuous-time process $(Y^N_{Nt})_{t \in \R_+}$ converges in distribution to a Brownian motion as $N \to \infty$ (see e.g.\ \cite{Donsker51}, \cite{Prokhorov56} or Section~8.6 in~\cite{durrett}).

In this paper we address the question of whether we can choose the scale factor $a_N$ in such a way that the scaled random walk $(Y^N_{Nt})_{t \in \R_+}$ converges in distribution to a time homogeneous diffusion $M$ satisfying the stochastic differential equation (SDE)
\ben\label{introSDE}
dM_t = \eta(M_t) dW_t, \quad M_0 = m,
\een
where $W$ is a Brownian motion, and
$\eta\colon \R \to \R$ is a Borel-measurable function
that satisfies
the Engelbert-Schmidt conditions
(see~\cite{ES1985})
in some interval $(l,r)$, $-\infty\leq l<r\leq\infty$,
and vanishes outside $(l,r)$.

If convergence takes place, then one can use the limiting process $M$ as a proxy
for the scaled random walk $Y^N$ for large $N$; or vice versa, $Y^N$ can be used for approximating the SDE $M$. One can thus profit from tools for continuous-time {\it and} discrete-time processes for analyzing both processes $M$ and $Y^N$.

If $\eta$ is Lipschitz continuous, then a natural choice for the scale factor is $a_N(y) = \frac{1}{\sqrt{N}} \eta(y)$. Then $(Y^N_k)$ can be interpreted as the Euler approximation of $M$, and it is known that it converges in distribution to $M$ (see e.g.\ \cite{kloeden1992numerical}). For {\it arbitrary} diffusion coefficients $\eta$ satisfying the Engelbert-Schmidt conditions the question of whether there exist scale factors such $Y^N$ converges to $M$ has not been solved. If the diffusion coefficient is very irregular, then the diffusion intensity $\eta(x_0)$ at a fixed state point $x_0$ can not be used as an approximation of the diffusion coefficient in the neighborhood of $x_0$. Therefore, in order to have convergence, the scaling factors $a_N$ need to take into account the global structure of $\eta$.

Recall that Skorokhod proves Donsker's theorem by embedding in law the random walk scaled by the constant $\frac{1}{\sqrt{N}}$ into the Brownian motion with a sequence of stopping times (see~\cite{skorokhod65}).
We take on Skorokhod's idea and show, under some nice conditions on the distribution of $X_i$,
that there exists a scale factor $a_N\colon (l,r) \to (0,\infty)$ such that $(Y^N_k)_{k \ge 0}$ can be embedded into the diffusion $M$ with a sequence of stopping times with expected time step $1/N$.

Loosely speaking, the embedding works as follows.
We first choose $a_N(m)$
(recall that $m$ is the starting point
in~\eqref{introSDE})
and a stopping time $\rho_1$ such that  $E(\rho_1) = 1/N$ and $M_{\rho_1} \stackrel{d}{=} Y^N_1$.
Conditionally on $\{M_{\rho_1} = y\}$ we choose $a_N(y)$ and a stopping time $\rho_2$ such that $E(\rho_2) = 1/N$ and $M_{\rho_1 + \rho_2} \stackrel{d}{=} y + a_N(y) X_2$. By proceeding like this we obtain a sequence of stopping times $\tau_k =\rho_1 + \ldots + \rho_k$ such that $(M_{\tau_k})_{k \ge 0}$ has the same distribution as the scaled random walk $(Y^N_k)_{k \ge 0}$.

The times $\rho_k$ turn out to be pairwise
uncorrelated and we can check that
they satisfy a certain uniform integrability
property (see Lemma~\ref{lemma ui}).
Under such a uniform integrability property
we prove a version of the weak law of large
numbers for uncorrelated arrays,
which is also interesting in itself
because we do not require finiteness
of the second moments
(see Theorem~\ref{lem:LLN}).
This weak law of large numbers entails that for all $t \in \R_+$ we have $\tau_{\lfloor Nt\rfloor} \to t$ in probability, as $N \to \infty$. From this, one can deduce that $(M_{\tau_{\lfloor Nt\rfloor}})_{t \in \R_+}$ converges in probability to $M$
uniformly on compact time intervals.
Therefore, $(Y^N_{Nt})_{t\in\R_+}$
converges in distribution to~$M$.

For our approach to work one needs to make sure that for every $N \in \N$ and $y \in (l,r)$ there exists a scale factor $a_N(y)$ such that the distribution of $y + a_N(y) X_i$ can be embedded into the diffusion $M$, conditioned to $M_0 = y$, with a stopping time with expectation $1/N$. The collection of distributions that can be embedded into $M$ with integrable stopping times is fully described in~\cite{AHS}. Moreover, there is a closed form integral expression for the minimal expectation of an embedding stopping time (see Theorem~3
in~\cite{AHS}).
This allows us to derive
%necessary and sufficient
weak sufficient conditions (see Section \ref{sec3}) for the existence of a scale factor
$a_N\colon(l,r) \to (0,\infty)$
such that $(Y^N_k)$ can be embedded into $M$ with stopping times having expectation~$1/N$.

Our approach to generalize Donsker’s theorem is essentially different from the one pioneered by Stone in~\cite{Stone63}
(also see~\cite{ALW} for a recent
generalization to tree-valued processes).
In that approach the approximating processes
are \emph{continuous-time} Markov processes that
\emph{do not jump over points in their state spaces}
(that is, they can be e.g.\ diffusions
or birth and death processes).
On the contrary, in this paper
we approximate $M$ via \emph{discrete-time}
Markov chains.
Another conceptual difference
is that we develop our theory without
requiring that the approximating Markov chains
do not jump over points.
On an informal level, one might view conditions
\eqref{cond a2l}--\eqref{cond a3lrinfty}
and \eqref{cond:emc1}--\eqref{cond:emc4}
at which we arrive in Section~\ref{sec3}
as an indication of what
comes out when we want
to allow overjumping.

%\red{In what concerns the reference measure $\mu$,
%our assumptions include the implications
%$l>-\infty\Rightarrow\inf\supp\mu>-\infty$
%and
%$r<\infty\Rightarrow\sup\supp\mu<\infty$,
%that is, our method cannot be used with say
%a normal distribution $\mu$ for SDEs
%defined on the half line.
%We will see that such a convergence result always
%holds whenever $\mu$ has a compact support
%in $\bbR$ and the infimum and supremum
%of its support are atoms of $\mu$, that is,
%$\mu(\{\inf\supp\mu\})>0$ and
%$\mu(\{\sup\supp\mu\})>0$.
%Weaker but more technical assumptions on $\mu$
%are discussed in Section~\ref{sec3}.}
%In what concerns the diffusion coefficient $\eta$, we only assume that it is Borel measurable
%and locally bounded away from $0$ and from $\pm\infty$ on the state space; we do {\it not} impose any regularity or growth assumption. In particular, our approach can be used for approximating SDEs with non-Lipschitz, or even discontinuous coefficients. E.g.\ it can be used for approximating diffusions in discontinuous media, diffusions with constant elasticity of variance (CEV), Wright-Fisher diffusions, etc.

The paper is organized in the following way. In Section~\ref{sec1}, we recall a necessary and sufficient condition, derived in~\cite{AHS}, for a distribution to be embeddable into the diffusion $M$ with an {\it integrable} stopping time. Moreover, we slightly generalize an integral formula for the minimal expectation of an embedding stopping time. In Section~\ref{sec3}, we characterize families of scaled random walks whose laws can be embedded into $M$ via a sequence of increasing stopping times such that the expected distance between two consecutive stopping times is equal to $1/N$, for $N \in \N$. In Section~\ref{sec weak conv}, we provide sufficient conditions for a sequence of scaled random walks, embeddable into $M$, to converge in distribution to~$M$.

%==========
\section{Embedding distributions in integrable time}\label{sec1}
In this section we recall a necessary and sufficient condition from \cite{AHS} for a centered distribution to be embeddable in a diffusion via an integrable stopping time.

Let $I = (l,r)$ with $l \in [-\infty,\infty)$ and $r  \in (-\infty, \infty]$.
As usual we denote by $\bar I$ the closure of $I$ in~$\R$.
Let $\eta: \R \to \R$ be a Borel-measurable function satisfying
\begin{align}\label{eq:eiit1}
\eta(x) & \ne 0 \text{ for all } x \in I, \\
\label{eq:eiit2}
\frac{1}{\eta^2} & \in L^1_\loc(I), \\
\label{eq:eiit3}
\eta(x) & = 0 \text{ for all } x \in \R\setminus I,
\end{align}
where $L^1_\loc(I)$ denotes the set of functions that are locally integrable on $I$.

Consider the SDE
\ben\label{sde}
dM_t = \eta(M_t) dW_t, \quad M_0 \sim \gamma,
\een
where $\gamma$ is a probability measure on $I$. The assumptions \eqref{eq:eiit1}--\eqref{eq:eiit3} imply that \eqref{sde} possesses a weak solution that is unique in law (see e.g.\ \cite{ES1985} or Theorem 5.5.7 in \cite{KS}). This means that there exists a pair of processes $(M,W)$ on a filtered probability space $(\Omega, \cF, (\cF_t), P)$, with $(\cF_t)$ satisfying the usual conditions, such that $W$ is an $(\cF_t)$-Brownian motion, $M_0$ is an $\cF_0$-measurable random variable with distribution $\gamma$ and $(M,W)$ satisfies the SDE~\eqref{sde}.
Let us note that $M$ stays in~$l$ (resp.~$r$)
once it hits~$l$ (resp.~$r$).

For all $y \in I$ and $x \in \R$ we define
\be
q(y,x) = \int_y^x \int_y^u \frac{2}{\eta^2(z)}\,dz\,du.
\ee
Notice that It\^o's formula implies that the process $(q(M_0,M_t)-t)$
is a local martingale starting in~$0$.
The assumptions \eqref{eq:eiit1}--\eqref{eq:eiit3} imply that for all $y \in I$ the nonnegative function $q(y, \cdot)$ is finite on $I$ and equal to $\infty$ on $\R \setminus [l,r]$. Besides, $q(y,\cdot)$ is strictly convex on $I$, strictly decreasing to zero on $(l,y)$ and strictly increasing from zero on $(y,r)$. Moreover, for all $y, \bar y \in I$ and $x \in \R$ we have
\ben\label{eq:eiit10}
q(y,x) = q(\bar y,x) - q(\bar y,y) - q_x(\bar y, y) (x-y),
\een
where $q_x$ denotes the partial derivative of $q$ with respect to the second argument.

Recall that by Feller's test for explosions we have $q(y,l+) = \infty$ if and only if the probability for the process $M$ to attain $l$ in finite time is equal to zero. Notice that the non-explosion condition  $q(y,l+) = \infty$ does {\it not} depend on $y$. Moreover, if $l=-\infty$, then $q(y, l+) = \infty$, and hence any solution of \eqref{sde} does not attain $-\infty$ in finite time. Similar statements hold true for the right-hand side boundary $r$.

We next recall a result from \cite{AHS} providing a necessary and sufficient condition for a distribution to be embeddable in $M$ with an integrable stopping time. Let $\mu$ be a centered probability measure on $\R$,
i.e.\ $\int |x|\,\mu(dx) < \infty$ and $\int x\,\mu(dx) = 0$.
Moreover, we assume that $\mu \ne \delta_0$.
Let $K(y,a,\cdot)$, $y\in\bar I$, $a\in\R_+$, be the location-scale family of probability distributions defined by
\ben\label{eq:eiit6}
K(y,a,B) = \mu\left( \frac{B-y}{a}\right), \quad B \in \cB(\R),
\een
whenever $a>0$; and $K(y,0,\cdot)=\delta_y$.
Consider the problem of finding a stopping time $\tau$
such that
\ben\label{eq:eiit7}
\Law(M_\tau |\cF_0) = K(M_0, a(M_0), \cdot),
\een
where $\Law(M_\tau |\cF_0)$ denotes the regular conditional distribution of $M_\tau$ with respect
to~$\cF_0$ and $a\colon \bar I \to \bbR_+$
is a given Borel function.
The unconditional version of this problem is usually referred to as the \emph{Skorokhod embedding problem}
or the \emph{SEP},
see~\cite{hobson2011} or~\cite{Obloj2004} for an overview.
In the subsequent sections we need embedding stopping times that are integrable conditionally on~$\cF_0$,
i.e.\ that satisfy $E[\tau | \cF_0] < \infty$~a.s.
For all $y \in I$ we define
\ben\label{eq:eiit8}
Q(y) = \int_\R q(y,x) K(y, a(y), dx).
\een
One can show that $Q(y)$ is the minimal expected time required for embedding $K(y, a(y), \cdot)$ into $M$, conditional to $M_0 = y$. To provide an intuition, suppose that
the starting in~$0$ local martingale
$(q(M_0,M_t)-t)$ is a true martingale and $\tau$ is a solution of the embedding problem~\eqref{eq:eiit7}. If the optional sampling theorem applies, then $E[\tau|\cF_0]=E[q(M_0,M_\tau)|\cF_0]=Q(M_0)$.
More formally, we have the following result, which is a straightforward
generalization of Theorem~3 and Proposition~4 in~\cite{AHS}:
\begin{theo}\label{th:eiit}
(i) Any $(\cF_t)$-stopping time $\tau$
solving~\eqref{eq:eiit7} satisfies
$E[\tau|\cF_0] \ge Q(M_0)$~a.s.

\smallskip
(ii) There exists a solution $\tau$ of the embedding problem~\eqref{eq:eiit7}
satisfying the property $E[\tau|\cF_0]<\infty$~a.s.
if and only if
\ben\label{eq:100914a2}
Q(M_0) < \infty
\quad\text{a.s.}
\een
In this case, there exists an embedding stopping time $\tau$ with
\ben\label{eq:100914a3}
E[\tau|\cF_0] = Q(M_0)
\quad\text{a.s.}
\een
\end{theo}

For the proof of the main results of Section \ref{sec weak conv} it turns out to be helpful to work with the specific solution of the embedding problem~\eqref{eq:eiit7} provided in~\cite{AHS}.
For the reader's convenience we briefly explain the solution method in the Appendix.

%==========
\section{Embedding scaled random walks}\label{sec3}
Let $(M,W)$ be a weak solution of
\ben\label{sdeM}
dM_t = \eta(M_t) dW_t, \quad M_0 =m,
\een
with $m \in I$.
Moreover, let $X_1, X_2, \ldots$ be a sequence of i.i.d.\ real-valued integrable random variables with $E(X_i) = 0$. We denote the distribution of $X_i$ by $\mu$.
Throughout we assume that $\mu\ne\delta_0$.

\begin{defi}\label{def:celsf1}
Let $a:\R \to \R_+$ be a Borel function. The process $Y = (Y_k)_{k \in \IZ_+}$, defined by $Y_0 = m$ and
\ben\label{eq:celsf1}
Y_{k+1} = Y_k + a(Y_k) X_{k+1}, \quad k \ge 0,
\een
is called {\it random walk generated by $(X_k)$ with scale factor $a$ and starting point $m$.}

We say that $Y = (Y_k)_{k \in \IZ_+}$ is a {\it scaled random walk} if there exists a scale factor $a$ such that $Y$ is the random walk generated by $(X_k)$ with scale factor $a$.
\end{defi}

In this section we aim at constructing scale factors such that $(Y_k)$ can be embedded in distribution into $M$ with a sequence of stopping times $(\tau_k)$ such that $(M_{\tau_k}) \stackrel{d}{=}(Y_k)$,
that is, both discrete time processes have the same law.
More precisely, we solve the following problem.

\medskip\noindent
\textbf{Problem~(P).}
Let $N \in \IN$. Does there exist a scale factor $a_N$ such that the associated scaled random walk $(Y^N_k)_{k \in \IZ_+}$ with $Y^N_0=m$ can be embedded in distribution into $M$ with a sequence of $(\cF_t)$-stopping times $(\tau^N_k)_{k \in \IZ_+}$ with
\ben\label{mean time step}
\tau^N_0=0
\quad\text{and}\quad
E[\tau^N_{k+1}| \cF_{\tau^N_k}] = \tau^N_k + \frac{1}{N},
\een
for all $k\geq0$?

\medskip
In order to determine the scale factor solving
Problem~(P), we introduce, for all $y \in I$, the mapping $G_y:[0,\infty) \to [0,\infty]$ defined via
\ben\label{defi Gy}
G_y(a) = \int_\R q(y,x)\,K(y,a,dx)
= \int_\R q(y,y+ax)\,\mu(dx).
\een
Recall that $G_y(a)$ is the minimal expected time needed for embedding $K(y,a,dx)$ into $M$ (cf.\ Theorem \ref{th:eiit} and the discussion following~\eqref{eq:eiit8}).
Notice that, for all $y\in I$, the map $G_y(\cdot)$ is strictly increasing with $G_y(0)=0$,
left-continuous by the monotone convergence theorem,
and continuous on
$[0,\infty)\setminus\{a_{\textrm{inf}}\}$ with
$$
a_{\textrm{inf}}:=\inf\{a\in[0,\infty):G_y(a)=\infty\},
\quad\inf\emptyset:=\infty,
$$
by the dominated convergence theorem.

We now provide sufficient conditions guaranteeing that a solution of Problem~(P) exists. We need to distinguish four cases.

%==========
\subsection{Case 1: $l = -\infty$ and $r = \infty$}
\label{subsec:2306a1}
In this subsection we make the following assumption.
\bi
\item[(A1)] There exists $y\in I$ such that $G_y(a)<\infty$ for all $a>0$.
\ei

\begin{lemma}\label{Eq condi A1}
(A1) is equivalent to the condition that for all $y \in I$ and $a >0$ we have $G_y(a)<\infty$.
\end{lemma}

\begin{proof}
Let $\bar y \in I$ and suppose that $G_{\bar y}(a)<\infty$ for all $a>0$. Let $y \in I$ and notice that
$q(y,x) = q(\bar y,x) - q(\bar y,y) - q_x(\bar y,y) (x-y)$.
Since $\mu$ is centered, we have
\be
G_y(a)=
\int_{\bbR} q(y,y+ax)\,\mu(dx) = \int_{\bbR} q(\bar y, y+ax)\,\mu(dx) - q(\bar y, y).
\ee
%Observe that $p_y$ is strictly convex and has a global minimum point at $0$. Thus it is strictly increasing to right of $0$, and strictly decreasing to its left.
For all $a >0$ and $x \in \R$ with $|x| \ge \frac{|y-\bar y|}{a}$ we have
\be
q(\bar y, y+ax) \le q(\bar y, \bar y + 2a x).
\ee
From this we obtain $G_y(a) < \infty$.
\end{proof}

The following theorem provides a solution to
Problem~(P) in Case~1.

\begin{theo}\label{existence_case1}
If (A1) is satisfied, then for all $N \in \IN$ there exists a unique scale factor $a_N$
satisfying
\ben\label{sf eq}
G_y( a_N(y)) = \frac{1}{N}, \quad y \in I.
\een
Moreover, the random walk $(Y^N_k)_{k \in \IZ_+}$ generated by $(X_k)$ with scale factor $a_N$ and starting point $m$ can be embedded into $M$ with a sequence of stopping times satisfying \eqref{mean time step}.
\end{theo}

\begin{remark}
It is worth noting that (A1) is satisfied
whenever $\mu$ has compact support.
\end{remark}

For the proof of the theorem we need the following auxiliary result.

\begin{lemma}\label{G bij}
If (A1) is satisfied, then $G_y$ is a bijective mapping from $[0,\infty)$ to $[0,\infty)$, for all $y\in I$.
\end{lemma}
\begin{proof}
Notice that $\lim_{x \to \pm \infty} q(y,x+y) = \infty$. Moreover, if $0 \le a < b$ and $x \ne 0$, then $q(y,y+ax) <q(y,y+bx)$. Therefore, $G_y$ is strictly increasing and by monotone convergence, $\lim_{a\to \infty} G_y(a) = \infty$.  Condition (A1), Lemma \ref{Eq condi A1} and a dominated convergence argument show that $G_y$ is continuous, and consequently, bijective.
\end{proof}

\begin{proof}[Proof of Theorem \ref{existence_case1}]
Let $N \in \IN$. Lemma \ref{G bij} implies that for all $y \in I$ there exists a scale factor $a_N(y)$ that satisfies \eqref{sf eq}.

%Consider the Markov kernel defined by $\nu_N(y,B) = K(y,a_N(y),B)$, for $y \in I$ and $B \in \cB(\R)$. By standard existence results (see e.g.\ Chapter 8 in \cite{shiryaev}) there exists a Markov chain $(Y^N_n)_{n \in \IZ_+}$, starting in $m$, such that $\Law(Y_{n+1}^N|Y^N_n) = \nu_N(Y^N_n, \cdot)$ for all $0\le n\le N$.

We next define a sequence of stopping times $(\tau(k))_{k \in \IZ_+}$ that embeds the transition probabilities into the diffusion $M$. First define $\tau(0)= 0$. Suppose that $\tau(k)$ is already defined. Set $X_t = M_{t+ \tau(k)}$, for $t \ge 0$, and observe that
\be
dX_t = \eta(X_t) \, d(W_{t+\tau(k)} - W_{\tau(k)}), \quad X_0 = M_{\tau(k)}.
\ee
Theorem \ref{th:eiit} implies that there exists an $(\cF_{t+\tau(k)})$-stopping time $\rho(k+1)$ with
\be
E[\rho(k+1)|\cF_{\tau(k)}] = Q(M_{\tau(k)}) = G_{M_{\tau(k)}}(a_N(M_{\tau(k)})) = \frac{1}{N}
\ee
such that $X_{\rho(k+1)} \stackrel{d}{=} Y^N_k + a_N(Y^N_k) X_{k+1}$.
%\be
%\Law(X_{\rho(k+1)}|\cF_{\tau(k)})=\nu_N(M_{\tau(k)},\cdot).
%\ee
Now define $\tau(k+1) = \tau(k) + {\rho(k+1)}$. By construction, the sequence $(M_{\tau(k)})_{k \in \IZ_+}$ has the same distribution as $(Y^N_k)_{k \in \IZ_+}$.
\end{proof}

The next example shows that a scale factor
satisfying~\eqref{sf eq} does not necessarily exist  if (A1) does not hold true.

\begin{ex}
\label{ex:celsf1}
Let $\mu$ be the probability measure with density $f(x) = c \frac{e^{-|x|}}{1+x^2}$, where $c = \left( \int\frac{e^{-|x|}}{1+x^2} dx \right)^{-1}$. Moreover let $m = 0$ and $\eta(x) =  e^{-x/2}$, $x \in \R$. Then we have $q(y,x)=2e^y(e^{x-y}-(x-y)-1)$. A straightforward calculation shows that $G_y(a) = \infty$ for $a > 1$. Therefore Condition (A1) is not satisfied. Moreover, for $a=1$ we have that
\be
G_y(1)=2e^yc\int_\R (e^x-x-1)\frac{e^{-|x|}}{1+x^2}dx<\infty.
\ee
By considering the limit $y\to -\infty$ we see that for every $N\in \IN$ there exists $y\in \IR$ such that $G_y(1)<1/N$. In particular, there exists no solution to \eqref{sf eq}.
\end{ex}

%==========
\subsection{Case 2: $l > -\infty$ and $r = \infty$}
\label{subsec:2306a2}
Here we impose the following assumption.
\bi
\item[(A2)] $\inf \supp \mu>-\infty$ and there exists $y\in I$ such that the integral over the positive real line $\int_{\bbR_+} q(y,y+ax)\,\mu(dx)<\infty$ for all $a>0$.
\ei

For every $y>l$ we set $\bar a(y)=\frac{l-y}{\inf \supp \mu}$. Note that for all $a\le \ol{a}(y)$ we have $a \supp \mu\subset [l-y,\infty)$. In the following we use the short-hand notation $q(x) = q(m,x)$, $x \in \R$.

We now present a solution to Problem~(P) in Case~2.

\begin{theo}\label{existence_case2}
Suppose that (A2) is satisfied and additionally that the following implications hold true:
\ben
&\text{if } q(l+)<\infty,& \text{ then } \mu(\{\inf \supp \mu\})>0, \label{cond a2l}\\
&\text{if } q(l+)=\infty,& \text{ then }
\liminf_{y\to\infty}G_y(\bar a(y))>0
\text{ and }
\liminf_{y\searrow l}G_y(\bar a(y))>0.
\label{cond a3lrinfty}
\een
Then there exists $N_0\in \IN$ %\footnote{$N_0=\lceil 1/\inf_{y\in (l,\infty)}G_y(\bar a(y)) \rceil$}
such that for all $N \ge N_0$ there exists a unique scale factor $a_N$
satisfying
\ben\label{sf eq2}
a_N(y) = \sup\left\{a \in [0, \infty): G_y(a) \le \frac{1}{N} \right\}, \quad y \in I,
\een
and $a_N(l)=0$.
Moreover, the random walk $(Y^N_k)_{k \in \IZ_+}$, generated by $(X_k)$ with scale factor $a_N$ and starting in $m$, can be embedded in $M$ with stopping times satisfying~\eqref{mean time step}.

In the case $q(l+)<\infty$, we can take $N_0=1$;
while in the case $q(l+)=\infty$,
the scale factor $a_N$ of~\eqref{sf eq2}
satisfies~\eqref{sf eq} for all $y\in I$ and $N\ge N_0$
(here $N_0\ge1$ can be necessary).
\end{theo}

\begin{remark}
It is worth noting that the assumptions
of Theorem~\ref{existence_case2} are satisfied
whenever $\mu$ has compact support and
$\mu(\{\inf\supp\mu\})>0$
(see Proposition~\ref{suffcond_1}).
\end{remark}

\begin{proof}
From similar arguments as in the proof of Lemma~\ref{Eq condi A1} it follows that condition~(A2) implies $\int_{\bbR_+} q(y,y+ax)\,\mu(dx)<\infty$ for all $y\in I$ and $a>0$. Notice that the sup in~\eqref{sf eq2} is attained, since $G_y$ is left-continuous. As in the proof of Lemma~\ref{G bij} one can show that $G_y:[0,\bar a(y)]\to [0,G_y(\bar a(y))]$ is bijective.

Now assume $q(l+)<\infty$. \eqref{cond a2l} implies that $w = \mu(\{\inf \supp \mu\})$ is positive. Let
$N\in\N$ and
$\nu_N(y, B) = K(y,a_N(y),B)$ for $y \in I$ and $B \in \cB(\R)$. Similarly to the proof of Theorem \ref{existence_case1} we construct a sequence of stopping times $(\tau(k))_{k \in \IZ_+}$ embedding the transition probabilities into $M$. Let $\tau(0)= 0$. Suppose that $\tau(k)$ is defined. %Let $Q(y) = G_y( a_N(y))$.
By Theorem \ref{th:eiit} there exists a stopping time $\rho(k+1)$ that embeds $\nu_N(M_{\tau(k)}, \cdot)$ into the process $X_t = M_{t + \tau(k)}$, $t \ge 0$, and satisfies
\be
E[\rho(k+1)|\cF_{\tau(k)}] = Q(M_{\tau(k)}) = G_{M_{\tau(k)}}(a_N(M_{\tau(k)})) \le \frac{1}{N}.
\ee
Define $\tau(k+1) = \tau(k) +\rho(k+1)$ if $M_{\tau(k) + \rho(k+1)} > l$, and $\tau(k+1) = \tau(k) +\rho(k+1) + \frac{1}{w}[\frac{1}{N} - Q(M_{\tau(k)})]$ if $M_{\tau(k) + \rho(k+1)} = l$. Then we have
\be
E[\tau(k+1) | \cF_{\tau(k)}]
%&=& \tau(k) + E\left[\rho(k+1) + \frac{1}{w}[\frac{1}{N} - c(M(\tau(k)))]1_{\{M_{\tau(k) + \rho(k+1)} = l \}} | \cF_{\tau(k)} \right] \\
&=& \tau(k) + Q(M_{\tau(k)}) + \frac{1}{w}\left[\frac{1}{N} - Q(M_{\tau(k)})\right]P(M_{\tau(k) + \rho(k+1)} = l | \cF_{\tau(k)}) \\
&=& \tau(k) + \frac{1}{N}.
\ee

Next assume that $q(l+)=\infty$.
Due to~\eqref{cond a3lrinfty} and
Lemma~\ref{lem:0706a1} below,
we have
$$
\inf_{y\in I}G_y(\bar a(y))>0.
$$
Choosing $N_0\in \IN$ such that $1/{N_0}<\inf_{y\in I}G_y(\bar a(y))$ yields that for every $N\ge N_0$ and $y\in I$ we have $G_y(a_N(y)) = \frac{1}{N}$. The rest of the proof goes along the lines of the proof of Theorem~\ref{existence_case1}.
\end{proof}

\begin{lemma}
\label{lem:0706a1}
Suppose (A2).

(i) The function $y\mapsto G_y(\bar a(y))$
is a lower semicontinuous function
$I\to(0,\infty]$.

(ii) For any compact subinterval
$J\subset I$, we have
$$
\inf_{y\in J}G_y(\bar a(y))>0.
$$
\end{lemma}

\begin{proof}
To simplify notation we assume that $\inf\supp\mu=-1$.
For $y>l$, we have $G_y(\bar a(y))>0$,
$\bar a(y)=y-l$ and
$g(y):=G_y(\bar a(y))=\int q(y,y+(y-l)x)\,\mu(dx)$.
Since $q(y,\cdot)$ nonnegative,
Fatou's lemma yields for $y_0 \in I$
\begin{equation}\label{eq:fatou}
\liminf_{y\to y_0} g(y)\ge \int \liminf_{y\to y_0}
q(y,y+(y-l)x)\,\mu(dx)=g(y_0).
\end{equation}
This proves the first statement.
The second statement
immediately follows from the first one.
\end{proof}

\begin{remark}
\label{rem:0706a1}
In comparison with Case~1
the condition that determines the scale factor
changes from~\eqref{sf eq}
to a less pleasant one~\eqref{sf eq2}.
Notice that in Case~1,
conditions \eqref{sf eq} and~\eqref{sf eq2} are equivalent.
As stated in Theorem~\ref{existence_case2},
in Case~2 with $q(l+)=\infty$
again \eqref{sf eq} holds true.
Example~\ref{ex:11102014a1} below shows that,
in Case~2 with $q(l+)<\infty$,
it can indeed happen that
the scale factor does not satisfy~\eqref{sf eq}
any longer.
However, by Lemma~\ref{lem:0706a1}~(ii),
the following statement holds:
\bi
\item[(Eq)]
Suppose (A2) and, for $N\in\bbN$,
consider the scale factor $a_N$ satisfying~\eqref{sf eq2}.
Then, for any compact subinterval $J\subset I$,
there exists $N_1\in\bbN$ such that,
for all $N\ge N_1$,
the scale factor $a_N$ satisfies~\eqref{sf eq}
on~$J$.
\ei
\end{remark}

\begin{ex}\label{ex:11102014a1}
Let $M$ be a Brownian motion starting at $m>0$
and absorbed as it reaches zero,
i.e.\ we have $I=(0,\infty)$
and $\eta\equiv1$.
Let $\mu=\frac12(\delta_{-1}+\delta_1)$.
A short computation shows that,
for $y>0$,
\be
G_y(a)=\begin{cases}
a^2&\text{if }a\in[0,y],\\
\infty&\text{if }a\in(y,\infty).
\end{cases}
\ee
Hence, $a_N(y)=\frac1{\sqrt{N}}\wedge y$,
and \eqref{sf eq} fails whenever
$y\in(0,\frac1{\sqrt{N}})$.
\end{ex}

While \eqref{cond a2l} is a condition on the primitives of our problem, \eqref{cond a3lrinfty} is harder to verify.
In the sequel we present sufficient conditions on $\mu$ (Proposition~\ref{suffcond_1}) and $\eta$ (Proposition~\ref{suffcond_2}) that imply~\eqref{cond a3lrinfty}.

\begin{propo}\label{suffcond_1}
Suppose (A2).
If $\mu(\{\inf \supp \mu\})>0$, then \eqref{cond a3lrinfty} is satisfied. Moreover, Theorem~\ref{existence_case2} applies with $N_0=1$.
\end{propo}

\begin{proof}
 From $q(l+)=\infty$ and $\mu(\{\inf \supp \mu\})>0$ it follows that $G_y(\bar a(y))=\infty$ for all $y\in I$, which implies the claims.
\end{proof}

\begin{propo}\label{suffcond_2}
Under (A2) assume that
$$
\limsup_{x\searrow l}\frac {|\eta(x)|}{x-l}<\infty
\quad\text{and}\quad
\limsup_{x\nearrow \infty}\frac {|\eta(x)|}{x}<\infty.
$$
Then \eqref{cond a3lrinfty} is satisfied.
\end{propo}

\begin{proof}
To simplify notation we assume that $\inf\supp\mu=-1$. For $y>l$ we have $\bar a(y)=y-l$ and $g(y):=G_y(\bar a(y))=\int h(x,y)\,\mu(dx)$
with $h(x,y):=q(y,y+(y-l)x)$.
We need to show that $\liminf_{y\to y_0} g(y)>0$ for $y_0\in \{l,\infty\}$. Note that
\begin{equation}\label{eq:int_rep_h}
 h(x,y)=2\int_0^x\int_0^u\frac{(y-l)^2}{\eta^2((y-l)z+y)}
 \,dz\,du.
\end{equation}
We have for every $z>-1$
\[
\liminf_{y\searrow l}\left(\frac{y-l}{|\eta((y-l)z+y)|}\right)=\frac 1{z+1}\liminf_{y\searrow l}\left(\frac{(y-l)z+y-l}{|\eta((y-l)z+y)|}\right)=\frac 1{z+1}\liminf_{x\searrow l}\left(\frac{x-l}{|\eta(x)|}\right)>0
\]
and
\[
\liminf_{y\nearrow \infty}\left(\frac{y-l}{|\eta((y-l)z+y)|}\right)=\frac 1{z+1}\liminf_{x\nearrow \infty}\left(\frac{x-l}{|\eta(x)|}\right)=\frac 1{z+1}\liminf_{x\nearrow \infty}\left(\frac{x}{|\eta(x)|}\right)>0.
\]
Thus, for $y_0\in \{l,\infty\}$, applying Fatou's lemma
in~\eqref{eq:int_rep_h} (observe that the area of integration is positively oriented also for $x\le 0$) yields $\liminf_{y\to y_0}h(x,y)>0$ for every $x\in(-1,\infty)\setminus\{0\}$.
Now the argument similar to~\eqref{eq:fatou}
yields the claim.
\end{proof}

\begin{remark}
We can replace the conditions
$\limsup_{x\searrow l}\frac {|\eta(x)|}{x-l}<\infty$ and $\limsup_{x\nearrow \infty}\frac {|\eta(x)|}{x}<\infty$
in the formulation of Proposition~\ref{suffcond_2}
by the weaker conditions $\liminf_{y\searrow l}h(\cdot,y)>0$ and $\liminf_{y\nearrow \infty}h(\cdot,y)>0$ on a set of positive mass with respect to $\mu$, where $h$ is defined as in the proof of
Proposition~\ref{suffcond_2}.
\end{remark}

Let us illustrate in more detail
how the assumptions in
Theorem~\ref{existence_case2}
work when $q(l+)=\infty$.
Recall that, in the case $q(l+)=\infty$,
the scale factor $a_N$ satisfies~\eqref{sf eq}
(not only~\eqref{sf eq2}).
If, however, (A2) does not hold true,
then a scale factor satisfying~\eqref{sf eq}
does not necessarily exist.
This can be shown by means of an example
similar to Example~\ref{ex:celsf1}.
The role of condition~\eqref{cond a3lrinfty} is
as follows.
Together with~(A2) it guarantees that,
in the case $q(l+)=\infty$,
there is a scale factor satisfying~\eqref{sf eq}.
Examples~\ref{ex:celsf2} and~\ref{ex:celsf3}
below show that \eqref{cond a3lrinfty}
can fail and a scale factor satisying~\eqref{sf eq}
does not necessarily exist
when we require (A2) alone.
%On the contrary, under~(A2), we do not need to have
%a solution to~\eqref{sf eq} in the case $q(l+)<\infty$.

\begin{ex}
\label{ex:celsf2}
Let us consider the constant elasticity of variance (CEV) process with \(\alpha>1\),
i.e. \(l=0\), \(r=\infty\) and \(\eta(x)=x^\alpha\)
on~\(I\).
For \(y>0\), we have
\ben\label{eq:celsf11}
q(y,x)=\begin{cases}
\frac2{2\alpha-1}\left[
\frac1{2\alpha-2}\left(
\frac1{x^{2\alpha-2}}
-\frac1{y^{2\alpha-2}}
\right)
+\frac{x-y}{y^{2\alpha-1}}
\right]&\text{if }x\geq0,\\
\infty&\text{if }x<0,
\end{cases}
\een
in particular, $q(y,0+)=\infty$.
We see that any centered measure
$\mu\ne\delta_0$
with $\inf\supp\mu>-\infty$
satisfies~(A2).

Notice that
$q\left(y/a,x/a\right)=a^{2\alpha-2}q(y,x)$
for $a,x,y>0$.
With $b=-\inf\supp\mu>0$
we now calculate
$\bar a(y)=y/b$ and
\be
G_y(\bar a(y))=\frac1{y^{2\alpha-2}}
\int_{\R}q\left(1,1+\frac xb\right)\,\mu(dx)
=\frac{b^{2\alpha-2}}{y^{2\alpha-2}}
\int_{\R}q\left(b,b+x\right)\,\mu(dx).
\ee
Thus, Theorem~\ref{existence_case2}
applies if and only if
$\mu\ne\delta_0$ is centered,
$\inf\supp\mu>-\infty$ and,
for some $\eps>0$,
\ben\label{eq:celsf10}
\int_{[\inf\supp\mu,\;\inf\supp\mu+\eps]}q\left(-\inf\supp\mu,-\inf\supp\mu+x\right)\,\mu(dx)=\infty.
\een
Here we used that such an integral
over $\R$ is infinite if and only if
\eqref{eq:celsf10} is satisfied
(for any $y>0$, the function
$x\mapsto q(y,x)$ has linear growth as $x\to\infty$).
Notice that a sufficient condition
for~\eqref{eq:celsf10}
is $\mu(\{\inf\supp\mu\})>0$.
\end{ex}

The previous example shows that,
choosing a centered measure
$\mu\ne\delta_0$
with $\inf\supp\mu>-\infty$
in a way that \eqref{eq:celsf10} fails,
we have (A2) but violate~\eqref{cond a3lrinfty}
in the way $G_y(\bar a(y))\to0$ as $y\to\infty$.
This raises the question of whether
it is possible to violate~\eqref{cond a3lrinfty},
under~(A2),
in the way $G_y(\bar a(y))\to0$ as $y\searrow l$.
This must be more delicate
because, on the one hand,
the condition
$|\eta(x)|\geq c(x-l)^\alpha$,
for all $x\in(l,b)$,
with some $c>0$, $b>l$ and $\alpha<1$,
implies $q(l+)<\infty$,
while, on the other hand, the condition
$|\eta(x)|\leq c(x-l)$,
for all $x\in(l,b)$,
with some $c>0$ and $b>l$, implies
$\liminf_{y\searrow l}G_y(\bar a(y))>0$
by Proposition~\ref{suffcond_2}.
Still this is possible as the following example shows.

\begin{ex}
\label{ex:celsf3}
We consider again $I=(0,\infty)$ and define
\[
\eta(x)=\begin{cases}
2\sqrt 2 x\frac{\textstyle(-\log x)^{\frac 34}}{\textstyle\sqrt{-1-2\log x}}&
\text{if }x\in(0,1/2),\\
1&\text{if }x\in[1/2,\infty).
\end{cases}
\]
Then one can verify that
\[
q(1/2,x)=\begin{cases}
\sqrt{-\log x}+\frac 1{\sqrt{\log2}}\left(x-\frac 12\right)-\sqrt{\log2}&
\text{if }x\in(0,1/2),\\
x^2-x+1/4&\text{if }x\in[1/2,\infty),
\end{cases}
\]
in particular, $q(1/2,0+)=\infty$.
We see that any centered measure
$\mu\ne\delta_0$ with $\inf\supp\mu>-\infty$
and $\int_{\R_+}x^2\,\mu(dx)<\infty$
satisfies~(A2).

\smallskip
Let now $\mu$ be a centered measure with
$$
\inf\supp\mu=-1
\text{ and }
\sup\supp\mu=1,
$$
in particular, $\bar a(y)=y$ for $y>0$.
Moreover assume that
\begin{equation}\label{eq:example_intcond}
\int_{\R}-\log\left(x+1\right)\,\mu(dx)<\infty.
\end{equation}
By formula~\eqref{eq:eiit10},
we have for $x\in(-1,1]$ and $y\in(0,1/4]$
\[
q(y,y+xy)=\sqrt{-\log[y(x+1)]}-\sqrt{-\log y}
+\frac{x}{2\sqrt{-\log y}}.
\]
In particular, for any $x\in(-1,1]$,
the mapping $y\mapsto q(y,y+xy)$ is increasing with $q(y,y+xy)\to 0$ as $y\searrow0$.
Indeed, we have
\[
 \sqrt{-\log[y(x+1)]}-\sqrt{-\log y}
 =\frac{-\log(x+1)}{\sqrt{-\log[y(x+1)]}+\sqrt{-\log y}}.
\]
Then dominated convergence (cf. \eqref{eq:example_intcond}) ensures that
$G_y(\bar a(y))\to 0$ as $y\searrow0$.
\end{ex}

Finally, we illustrate how
Theorem~\ref{existence_case2}
works when $q(l+)<\infty$.

\begin{ex}
Let us now consider the CEV process with
\(\alpha\in(-\infty,1)\setminus\{0\}\),
i.e. \(I=(0,\infty)\) and \(\eta(x)=x^\alpha\)
on~\(I\).
In the case $\alpha\neq\frac12$,
for \(y>0\),
the function $q(y,\cdot)$
is given by formula~\eqref{eq:celsf11}.
In the case $\alpha=\frac12$,
for \(y>0\), we have
\[
q(y,x)=\begin{cases}
2x\log\frac xy-2(x-y)&\text{if }x\geq0,\\
\infty&\text{if }x<0.
\end{cases}
\]
In particular, $q(y,0+)<\infty$ in both cases.
Thus, Theorem~\ref{existence_case2}
applies if and only if $\mu\ne\delta_0$
is centered, $\inf\supp\mu>-\infty$,
$\mu(\{\inf\supp\mu\})>0$ and
\be
&\text{if }\alpha=\frac12,&\text{then }
\int_{\R_+}x\log x\,\mu(dx)<\infty,\\
&\text{if }\alpha<\frac12,&\text{then }
\int_{\R_+}x^{2-2\alpha}\,\mu(dx)<\infty.
\ee
\end{ex}

%==========
\subsection{Case 3: $l =-\infty$ and $r < \infty$}

This case can be reduced to Case~2 by considering the diffusion $-M$.

%==========
\subsection{Case 4: $l > - \infty$ and $r < \infty$}

In this subsection we make the following assumption.
%For a distribution to be embeddable in $M$ it is necessary that
\bi
\item[(A3)] $\inf \supp \mu>-\infty$ and $\sup \supp \mu < \infty$.
\ei
For every $y\in I$ we set $\bar a(y)=\frac{l-y}{\inf \supp \mu} \wedge \frac{r-y}{\sup \supp \mu}$. Note that for all $a\le \ol{a}(y)$ we have $a \supp \mu\subset [l-y,r-y]$.
%Moreover, we denote by $c_0$ the unique real number in $I$ satisfying
%\be
%\frac{l-c_0}{\inf \supp \mu} = \frac{r-c_0}{\sup \supp \mu}.
%\ee

A solution to Problem~(P)
in Case~4 is given in the next theorem.

\begin{theo}\label{existence_case4}
Suppose that (A3) is satisfied and additionally that the following implications hold true:
\ben
&\text{if } q(l+)<\infty, & \text{ then } \mu(\{\inf \supp \mu\})>0, \label{cond:emc1}\\
&\text{if } q(l+)=\infty, & \text{ then }
\liminf_{y\searrow l} G_y(\bar a(y)) > 0, \label{cond:emc2}\\
&\text{if } q(r-)<\infty, & \text{ then } \mu(\{\sup \supp \mu\})>0, \label{cond:emc3}\\
&\text{if } q(r-)=\infty, & \text{ then }
\liminf_{y\nearrow r} G_y(\bar a(y)) > 0. \label{cond:emc4}
\een
Then there exists $N_0\in \IN$ %\footnote{$N_0=\lceil 1/\inf_{y\in (l,\infty)}G_y(\bar a(y)) \rceil$}
such that for all $N \ge N_0$ there exists a unique scale factor $a_N$ satisfying~\eqref{sf eq2} and $a_N(l)=a_N(r)=0$.
Moreover, the random walk $(Y^N_k)$, scaled with $a_N$ and starting in $m$, is embeddable in $M$ with stopping times $(\tau^N_k)_{k \in \IZ_+}$ satisfying~\eqref{mean time step}.
\end{theo}

\begin{proof}
Similar to the proof of Theorem \ref{existence_case2}.
\end{proof}

\begin{remark}
Notice that the assumptions
of Theorem~\ref{existence_case4} are satisfied
whenever $\mu$ has compact support,
$\mu(\{\inf\supp\mu\})>0$ and
$\mu(\{\sup\supp\mu\})>0$
(see Proposition~\ref{suffcond_3}).
\end{remark}

The next two propositions provide sufficient conditions for the properties \eqref{cond:emc2} and~\eqref{cond:emc4} to hold true.

\begin{propo}\label{suffcond_3}
Suppose (A3).
If $\mu(\{\inf \supp \mu\})>0$, then \eqref{cond:emc2} is satisfied.
\end{propo}

\begin{proof}
Similar to the proof of Proposition \ref{suffcond_1}.
\end{proof}

Similarly, the condition $\mu(\{\sup \supp \mu\})>0$ is sufficient for \eqref{cond:emc4}.

\begin{propo}\label{suffcond_4}
Suppose (A3).
If $\limsup_{x\searrow l}\frac {|\eta(x)|}{x-l}<\infty$, then \eqref{cond:emc2} is satisfied. If $\limsup_{x\nearrow r}\frac {|\eta(x)|}{r-x}<\infty$, then \eqref{cond:emc4} is satisfied.
\end{propo}

\begin{proof}
Similar to the proof of Proposition \ref{suffcond_2}.
\end{proof}

Finally, it is worth noting that
the detailed discussions in Case~2
about the role of different assumptions, etc.,
have their analogues in Case~4.
In particular, we have:
\bi
\item
The statements of Lemma~\ref{lem:0706a1} apply verbatim under (A3) instead of~(A2);
\item
Statement~(Eq) in Remark~\ref{rem:0706a1} applies verbatim under (A3) instead of~(A2);
\item
The conclusion of Theorem~\ref{existence_case4} holds true with $N_0=1$ whenever (A3) is satisfied and we have
$$
\mu(\{\inf\supp\mu\})>0
\quad\text{and}\quad
\mu(\{\sup\supp\mu\})>0
$$
(cf.~with the statements in the end of
Theorem~\ref{existence_case2} and
Proposition~\ref{suffcond_1});
\item
Under the assumptions of Theorem~\ref{existence_case4},
for all $N\ge N_0$ and $y\in I$,
the scale factors $a_N$ satisfy~\eqref{sf eq}
(not only~\eqref{sf eq2})
whenever $q(l+)=q(r-)=\infty$
(cf.~with the statement in the end of Theorem~\ref{existence_case2}).
\ei

%==========
\section{Weak convergence}\label{sec weak conv}
In this section we use the setting
and notations of Section~\ref{sec3}.
In particular, we consider a weak solution
$(M,W)$ of~\eqref{sdeM},
denote by $(Y^N_k)$ the scaled random walk \eqref{eq:celsf1}, and assume that $\mu\ne\delta_0$ is a centered probability measure on~$\bbR$.
Throughout this section we suppose that
one of the sufficient conditions from Section~\ref{sec3}
is satisfied that guarantees,
for sufficiently large $N\in\N$,
the existence of a scale factor $a_N$ satisfying~\eqref{sf eq2}
that solves Problem~(P).
Let us remark that under each of these
sufficient conditions, we have
\ben\label{int condi ref distri}
\int_\R q(m+ax)\,\mu(dx)<\infty
\quad\text{for some }a>0
\een
(recall that $q(x)$
is a short-hand notation for $q(m,x)$).
We extend $Y^N$ to a continuous-time process
on $\R_+$
via linear interpolation, i.e.\ for all $t\geq0$ we set $Y^N_t = Y^N_{\lfloor t \rfloor} + (t-\lfloor t \rfloor) (Y^N_{\lfloor t \rfloor + 1} - Y^N_{\lfloor t \rfloor})$.

In this section we show that if the diffusion
coefficient~$\eta$ is locally bounded
away from~$0$ and from~$\pm\infty$
and $\mu$ has compact support, then the sequence of continuous processes
$(Y^N_{Nt})_{t\geq0}$
converges in law
to the process $(M_t)_{t\geq0}$, as $N\to\infty$
(see Theorem~\ref{convi sums2}).
We also present other sets of sufficient conditions
for this weak convergence
(generally, the less we require on~$\eta$,
the more we need to require on~$\mu$).
One can thus interpret $(Y^N_k)_{k\in\IZ_+}$
as a Markov chain approximating the diffusion
$(M_t)_{t\geq0}$.

To simplify the analysis, we only show the weak convergence on the time interval $[0,1]$. A straightforward generalization implies the weak convergence on~$\R_+$.

We first assume that $\eta$ satisfies the following condition.
\bi
\item[(C1)] $|\eta|$ and $\frac{1}{|\eta|}$ are bounded on $I$.
\ei

\begin{theo}\label{convi sums}
Suppose that (C1) holds true.
Then the processes $(Y^N_{Nt})_{t\in[0,1]}$
converge to $(M_t)_{t \in [0,1]}$ in distribution,
as $N \to \infty$, i.e.\ the associated measures on
$\big(C[0,1], \cB(C[0,1])\big)$
converge weakly.
\end{theo}

We first show that
boundedness of $|\eta|$
implies that the scale factor
$a_N(y)$ is of order~$\frac1{\sqrt{N}}$.

\begin{lemma}\label{convscalefactor}
If $|\eta(x)|\le U<\infty$ for all $x\in I$,
then there exists $A\in \R_+$ such that $a_N(y) \le \frac{A}{\sqrt{N}}$
for all $N\in\N$ and $y \in I$.
\end{lemma}

\begin{proof}
For $x, y \in I$ we have
\ben\label{inequi 2202}
q(y,x) \ge \int_y^x \int_y^{u} \frac{2}{U^2}\,dz\,du = \frac{(x-y)^2}{U^2}.
\een
In particular, $q(x) \ge (x-m)^2/U^2$, and
hence~\eqref{int condi ref distri}
implies that $\int x^2\,\mu(dx) < \infty$.
It follows from \eqref{sf eq2} and \eqref{inequi 2202} that
\be
\frac 1N\ge G_y(a_N(y))=\int q(y,y+a_N(y)x)\,\mu(dx)\ge \frac{a_N^2(y)\int x^2\,\mu(dx)}{U^2},
\ee
which yields the claim with
$A = \frac{U}{\sqrt{\int x^2\,\mu(dx)}}$.
\end{proof}

\begin{lemma}\label{lemma ui}
Assume (C1).
The solution to Problem~(P)
can be chosen in such a way that
the $(\cF_{\tau^N(k-1)+t})$-stopping times
$\rho^N(k)=\tau^N(k)-\tau^N(k-1)$
have the following uniform integrability property:
the family $N \rho^N(k)$, $1\le k\le N$, $N\in\IN$,
is uniformly integrable.
\end{lemma}

\begin{proof}
%Suppose first that $|\eta| = \delta \in (0, \infty)$. Then $X_t := M_t - m$ is a Brownian motion multiplied with $\delta$. Let $\sigma$ be an integrable stopping time such that $X_\sigma \sim \mu$. For all $N\ge 1$ let $X^N_t = \frac{1}{\sqrt{N}} X_{Nt}$. Observe that $\sigma_N := \frac1N \sigma$ embeds $\mu\left( \sqrt{N} \cdot\right)$ into $X^N$. Since the process $X^N$ has the same distribution as $X$, we can thus find stopping times $\rho^N$ with $\rho^N \stackrel{d}{=} \sigma_N$ and $X_{\rho^N} \sim \mu\left( \frac{\cdot}{\sqrt{N}}\right)$.
%%%%%%%%%%%%%%%% old proof %%%%%%%%%%%%%%%%%%%
Choose $\rho^N(k)$ according to the construction method outlined in the Appendix. More precisely, suppose that $\rho^N(k)= \Delta(M_{\tau^N(k-1)})$ (see last line of the Appendix for the definition). We now show that the family $N \rho^N(k)$, $1\le k\le N$, $N\in\IN$, is uniformly integrable.

Below, for random variables
$\xi$ and $\zeta$,
we write $\xi\stackrel{d}{=}\zeta$
(resp. $\xi\stackrel{d}{\le}\zeta$)
to indicate that $\xi$ and $\zeta$
have the same distribution
(resp. $\zeta$ stochastically dominates~$\xi$).
Let $U<\infty$ be an upper bound for $|\eta|$. Then
Lemma~\ref{lemma st properties} yields
\be
\rho^N(1)=
\Delta(M_0) \stackrel{d}{=} \int_0^1 \frac{a_N^2(m) b^2_x(s,\tilde W_s)}{\eta^2(a_N(m) b(s,\tilde W_s) + m)} ds \ge \frac{a_N^2(m)}{U^2} \int_0^1 b^2_x(s,\tilde W_s) ds.
\ee
Therefore, the random variable
$X = \int_0^1 b^2_x(s,W_s) ds$ is integrable.
Now let $L>0$ be a lower bound for $|\eta|$. Then,
with $\tilde M_0$ having the distribution
of $M_{\tau^N(k)}$, we get
\be
\rho^N(k+1)=
\Delta(M_{\tau^N(k)}) \stackrel{d}{=} \int_0^1 \frac{a_N^2(\tilde M_0) b^2_x(s,\tilde W_s)}{\eta^2(a_N(\tilde M_0) b(s,\tilde W_s) + \tilde M_0)} ds  \stackrel{d}{\le}
\frac{A^2}{L^2 N}X,
\ee
where we also use that,
by Lemma~\ref{convscalefactor},
we have $a_N(y) \le \frac{A}{\sqrt{N}}$
for all $y\in I$.
Thus,
\be
P(N\rho^N(k+1) \ge x) \le P\left(\frac{A^2}{L^2}X \ge x\right), \quad x \ge 0.
\ee
In other words, the integrable random variable
$\frac{A^2}{L^2}X$ stochastically
dominates every $N\rho^N(k)$,
and hence we get the result.
\end{proof}

We next aim at showing that for $s \in [0,1]$ the stopping times $\tau^N(\lfloor Ns \rfloor)$ converge to $s$ in probability. To this end we use the following version of the weak law of large numbers.

\begin{theo}[Weak LLN for uncorrelated arrays]
\label{lem:LLN}
Let $(Z^n_k)_{n\in\bbN,1\leq k\leq n}$
be a triangular array of nonnegative
and uniformly integrable random variables.
Suppose that, for all $n\in\bbN$,
the collection $Z^n_k$, $1\leq k\leq n$,
is pairwise uncorrelated. Then
$\frac1n\sum_{k=1}^n (Z^n_k-E Z^n_k)$
converges to zero in probability.
\end{theo}

\begin{proof}
Let us set
$\ol Z^n_k=Z^n_k 1_{\{Z^n_k\leq n\}}$
and define the sums
$S_n=\sum_{k=1}^n Z^n_k$
and
$\ol S_n=\sum_{k=1}^n \ol Z^n_k$.
Since the family $(Z^n_k)_{n,k}$
is uniformly integrable, we have
\begin{align}
\label{eq:LLN1}
C&:=\sup_{n\in\bbN,1\leq k\leq n}E Z^n_k<\infty,\\
\label{eq:LLN2}
D(n)&:=\sup_{1\leq k\leq n}E Z^n_k
1_{\{Z^n_k>n\}}\to0,\quad n\to\infty.
\end{align}
We need to prove that
$\frac{S_n-E S_n}n$
converges to zero in probability.
It follows from~\eqref{eq:LLN2} and the estimates
\begin{align*}
 P(S_n\neq\ol S_n)
&\leq\sum_{k=1}^n  P(Z^n_k>n)
\leq n\sup_{1\leq k\leq n} P(Z^n_k>n)
\leq D(n),\\
0&\leq\frac{E S_n-E\ol S_n}n
=\frac1n\sum_{k=1}^n E Z^n_k
1_{\{Z^n_k>n\}}
\leq D(n)
\end{align*}
that it is enough to prove that
$\frac{\ol S_n-E\ol S_n}n$
converges to zero in probability.
We will now prove that the latter sequence
converges to zero in~$L^2$.

We have
\begin{align}
\notag
E\left(\frac{\ol S_n-E\ol S_n}n\right)^2
&=\frac1{n^2}\Var\ol S_n\\
%\notag
%&=\frac1{n^2}\left(\sum_{k=1}^n \Var\ol Z^n_k
%+2\sum_{1\leq k<l\leq n}\Cov(\ol Z^n_k,\ol Z^n_l)\right)\\
\label{eq:LLN3}
&\leq\frac1{n^2}\left(\sum_{k=1}^n E(\ol Z^n_k)^2
+2\sum_{1\leq k<l\leq n}\Cov(\ol Z^n_k,\ol Z^n_l)\right).
\end{align}
Due to the uniform integrability of $(Z^n_k)_{n,k}$,
$$
G(y):=\sup_{n\in\bbN,1\leq k\leq n}E Z^n_k
1_{\{Z^n_k>y\}}\to0,\quad y\to\infty.
$$
Since
$
E(\ol Z^n_k)^2
=\int_0^\infty 2y P(\ol Z^n_k>y)\,dy
\leq\int_0^n 2y P(Z^n_k>y)\,dy
\leq\int_0^n 2G(y)\,dy,
$
we get
\begin{equation}
\label{eq:LLN4}
\frac1{n^2}\sum_{k=1}^n E(\ol Z^n_k)^2
\leq\frac1n\int_0^n 2G(y)\,dy
\to0,\quad n\to\infty.
\end{equation}
Using that the random variables
$Z^n_k$, $1\leq k\leq n$,
are pairwise uncorrelated,
we get, for $k\neq l$,
\begin{align*}
\Cov(\ol Z^n_k,\ol Z^n_l)
&=E\ol Z^n_k\ol Z^n_l-E\ol Z^n_kE\ol Z^n_l
\leq E Z^n_k Z^n_l-E\ol Z^n_kE\ol Z^n_l\\
&=E Z^n_k E Z^n_l
-(E Z^n_k-E Z^n_k 1_{\{Z^n_k>n\}})
(E Z^n_l-E Z^n_l 1_{\{Z^n_l>n\}})\\
&\leq C(E Z^n_k 1_{\{Z^n_k>n\}}
+E Z^n_l 1_{\{Z^n_l>n\}})
\leq 2CD(n),
\end{align*}
where $C$ is the constant from~\eqref{eq:LLN1}.
Hence,
$\limsup_{n\to\infty}
\sup_{1\leq k<l\leq n}\Cov(\ol Z^n_k,\ol Z^n_l)=0$.
Together with~\eqref{eq:LLN4}
and the fact that the right-hand side of~\eqref{eq:LLN3}
is nonnegative, this implies that
the right-hand side of~\eqref{eq:LLN3}
converges to zero.
The proof is completed.
\end{proof}

Observe that \eqref{mean time step} implies that the sequence $(\rho^N(k))$ is pairwise uncorrelated and $E[\rho^N(k+1)|\cF_{\tau^N(k)}] = \frac{1}{N}$.

\begin{lemma}\label{convi st times}
Suppose that the family
$(N \rho^N(k))_{N\in\N,1\leq k\leq N}$
is uniformly integrable. Then
for all $s \in [0,1]$ we have
$\tau^N(\lfloor Ns \rfloor) \to s$
in probability.
\end{lemma}

\begin{proof}
Let $s \in [0,1]$. Set $Z^N_k = N \rho^N(k)$ if $k \le \lfloor Ns \rfloor$ and $Z^N_k = 0$ else. Notice that the family $(Z^N_k)_{1 \le k \le N}$ satisfies the assumptions of Theorem~\ref{lem:LLN}, and hence $\frac1N \sum_{k=1}^N (Z^N_k - E(Z^N_k))$ converges to zero in probability. Notice that
\be
\frac1N \sum_{k=1}^N Z^N_k = \sum_{k=1}^{\lfloor Ns \rfloor} \rho^N(k) = \tau^N(\lfloor Ns \rfloor),
\ee
and $\lim_{N \to \infty}\frac1N \sum_{k=1}^N E(Z^N_k) = \lim_{N \to \infty} \frac1N \lfloor Ns \rfloor = s$. Consequently, $\tau(\lfloor Ns \rfloor)$ converges to $s$ in probability, as $N \to \infty$.
\end{proof}

The final arguments for proving Theorem~\ref{convi sums} are now standard
(cf.\ Section~8.6 in~\cite{durrett}).
We denote by $\|\cdot\|_{C[0,1]}$
the sup norm in $C[0,1]$.

\begin{proof}[Proof of Theorem~\ref{convi sums}]
We can assume that $Y^N_k = M_{\tau^N(k)}$
and that the family
$N \rho^N(k)$, $N\in\N$, $1\leq k\leq N$,
is uniformly integrable
(see Lemma~\ref{lemma ui}). Recall that $Y^N_{Nt} = Y^N_{\lfloor Nt \rfloor} + (Nt - \lfloor Nt \rfloor) (Y^N_{\lfloor Nt \rfloor+1} - Y^N_{\lfloor Nt \rfloor})$ for $t \in [0,1]$.

First we show that $\|Y^N_{N\cdot} - M_\cdot\|_{C[0,1]} \to 0$ in probability.
To this end let $\eps > 0$. For $\delta > 0$  let
\be
A(\delta) = \{|M_t-M_s|  < \frac{\eps}{2} \text{ for all } t,s \in [0,1] \text{ such that } |t-s| \le 2\delta\}.
\ee
We choose $\delta$ such that $\frac{1}{\delta} \in \IN$ and $P(A(\delta)) > 1- \frac{\eps}{2}$. Next we define
\be
C(N,\delta) = \{|\tau^N(\lfloor Nk\delta \rfloor) - k\delta| \le \delta \text{ for } k=1, \ldots, \frac{1}{\delta} \}.
\ee
By Lemma~\ref{convi st times} there exists $N_0 \in \IN$ such that for all $N \ge N_0$ we have $P(C(N,\delta)) > 1-\frac{\eps}{2}$.

Notice that on the event $C(N,\delta)$ we have $|\tau^N(\lfloor Ns \rfloor) - s| \le 2\delta$ for all $s \in [0,1]$. In the following suppose that $A(\delta) \cap C(N,\delta)$ occurs. Then for $s=\frac{m}{N}$ we have $|\tau^N(m) - \frac{m}{N}| \le 2 \delta$ and hence
\be
|Y^N_{Ns} - M_s| = |Y^N_{m} - M_{\frac{m}{N}}| < \frac{\eps}{2}.
\ee
Let now $s \in (\frac{m}{N}, \frac{m+1}{N})$. Set $\theta = s-\frac{m}{N}$ and notice that for all $N \ge \frac{1}{2\delta}$
\begin{align*}
|Y^N_{Ns} - M_s| &\le \theta |Y^N_{m} - M_{\frac{m}{N}}| + (1-\theta) |Y^N_{m+1} - M_{\frac{m+1}{N}}|\\
&\hspace{1em}+ \theta|M_{\frac{m}{N}} -M_s| + (1-\theta)|M_{\frac{m+1}{N}} -M_s|
< \eps.
\end{align*}
Consequently, for all $N \ge N_0 \vee \frac{1}{2\delta}$ we have $P(\|Y^N_{N\cdot} - M_\cdot\|_{C[0,1]} > \eps) < \eps$. Since $\eps$ is arbitrary, we obtain that $\|Y^N_{N\cdot} - M_\cdot\|_{C[0,1]} \to 0$ in probability.

To complete the proof, let $\psi\colon C[0,1] \to \R$ be a bounded function that is continuous with respect to the sup norm. It is straightforward to show that $\lim_{N\to \infty} E\psi(Y^N_{N\cdot}) = E\psi(M_\cdot)$, and hence the theorem is proved.
\end{proof}

With a localization argument we can relax the assumption on~$\eta$:
\bi
\item[(C2)] $|\eta|$ and $\frac{1}{|\eta|}$ are locally bounded on $I$.
\ei

\begin{theo}\label{convi sums2}
Suppose (C2) and that
$\mu$ has a compact support.
Then the processes $(Y^N_{Nt})_{t\in[0,1]}$
converge to $(M_t)_{t \in [0,1]}$ in distribution,
as $N \to \infty$.
\end{theo}

\begin{proof}
The idea is first to redefine the times $\rho^N(k)$
to make sure that the family
$N \rho^N(k)$, $N\in\N$, $1\leq k\leq N$,
is uniformly integrable.
To this end choose a sequence of bounded intervals
$[l_n, r_n] \subset I$
such that $l_n \downarrow l$ and $r_n\uparrow r$.
Let $H(l_n,r_n)$ denote the first exit time
of $M$ from $(l_n,r_n)$.
For a fixed~$n$
let $\wh \rho^N(k+1) = \frac1N$
if $\tau^N(k) > H(l_n, r_n)$,
and let $\wh \rho^N(k+1) = \rho^N(k+1)$ otherwise.
We set $\wh \tau^N(k) = \sum_{j=1}^k \wh \rho^N(j)$.
Notice that $(\wh \tau^N(k))$ and $(\wh \rho^N(k))$
depend on the localizing parameter~$n$.

Next observe that $\lim_{N\to \infty} \sup_{y\in (l_n,r_n)}a_N(y)=0$. Indeed, by~\eqref{sf eq2}, we have
\ben\label{eq:18062015}
\int q(y,y+a_N(y)x)\,\mu(dx)\le\frac 1N.
\een
Fatou's lemma yields
$\int q(y,y+\liminf_{N\to\infty}a_N(y)x)\,\mu(dx)=0$.
Since $\mu\ne\delta_0$ and, clearly,
the sequence $\{a_N(y)\}_{N\in\N}$ is decreasing,
we obtain  $\lim_{N\to\infty}a_N(y)=0$.
Performing two changes of variables in \eqref{eq:18062015} leads to
\be
Na^2_N(y)\int\int_0^x\int_0^u\frac 2{\eta^2(y+a_N(y)r)}\,dr\,du\,\mu(dx) \le1.
\ee
Using Fatou's lemma again
and taking into account the
already established relation
$\lim_{N\to\infty}a_N(y)=0$,
we get
\be
\limsup_{N\to \infty}Na_N^2(y)\le \frac 1{\int\int_0^x\int_0^u{\liminf_{N\to \infty}} \frac  2{\eta^2(y+a_N(y)r)}\,dr\,du\,\mu(dx)}
\le
\frac{\limsup_{z\to y}\eta^2(z)}{\int x^2\,\mu(dx)}.
\ee
From (C2) we deduce that $\lim_{N\to \infty} \sup_{y\in (l_n,r_n)}a_N(y)=0$.

Since $\mu$ has a compact support
one can show by an adaptation
of Lemma~\ref{lemma ui}
that, for each fixed~$n$,
the family
$(N \wh \rho^N(k))_{N\in\N,1\le k\le N}$
is uniformly integrable.
Lemma~\ref{convi st times} implies
that $\wh \tau^N(\lfloor Ns \rfloor) \to s$ in probability
for all $s\in[0,1]$.
Let $\wh Y^N_k = M_{\wh \tau^N(k)}$
and $Y^N_k=M_{\tau^N(k)}$.
As in the proof of Theorem~\ref{convi sums}
one can show that
$\|\wh Y^N_{N\cdot} - M_\cdot\|_{C[0,1]} \to 0$
in probability.

Let us consider several cases.
First let $q(l+)=q(r-)=\infty$,
that is, both endpoints $l$ and~$r$
are inaccessible.
Fix $\eps>0$
and choose~$n$ such that
$P(H(l_n,r_n)\ge2)>1-\frac\eps2$.
Then we choose $N_0$ such that,
for any $N\ge N_0$, we have
$P(\wh\tau^N(N)\le2)>1-\frac\eps2$.
On the event $\{H(l_n,r_n)\ge2,\,\wh\tau^N(N)\le2\}$
of probability at least $1-\eps$
we have $Y^N_k=\wh Y^N_k$ for all $0 \le k \le N$.
Thus,
$\|Y^N_{N\cdot} - M_\cdot\|_{C[0,1]} \to 0$
in probability.

Let now $q(l+)<\infty$ and $q(r-)=\infty$,
i.e.\ $l$ is accessible, $r$ is inaccessible.
In this case, $l>-\infty$ and
$H(l_n,r_n)\to H(l)$~a.s.\ as $n\to\infty$,
where $H(l)$ denotes the hitting time of $l$
by the process~$M$.
Fix $\eps>0$ and choose~$n$
such that $P(A_n)<\frac\eps3$
and $P(B_n)<\frac\eps3$ with
\be
A_n&=&\left\{H(l)\le3\text{ and }
\exists\,s\in[H(l_n,r_n),H(l)]
\text{ such that }M_s>l+\eps\right\},\\
B_n&=&\left\{H(l)>3\text{ and }H(l_n,r_n)<2\right\}.
\ee
Then choose $N_0$ such that, for any $N\ge N_0$,
we have $P(\wh\tau^N(N)\le2)>1-\frac\eps3$.
Given an event~$A$, by $A^c$ we denote
the complement of~$A$.
For $N\ge N_0$, on the event
\be
\{\wh\tau^N(N)\le2\}\cap(A_n\cup B_n)^c
\ee
of probability at least $1-\eps$
we have either
\be
H(l)>3,\;H(l_n,r_n)\ge2,\text{ hence }
Y^N_k=\wh Y^N_k\text{ for all }0\le k\le N,
\ee
or
\be
&&H(l)\le3,\;M_s\in[l,l+\eps]
\text{ for }s\in[H(l_n,r_n),H(l)],\\
&&\text{hence }|Y^N_{Ns}-M_s|
\le\eps\text{ whenever }
Y^N_{Ns}\ne\wh Y^N_{Ns}, s\in[0,1].
\ee
Thus, $\|Y^N_{N\cdot}-M_\cdot\|_{C[0,1]}\to0$
in probability.

The remaining cases are considered in a similar way.
\end{proof}

We can also combine the boundedness assumptions
on $\eta$ and on the support of $\mu$ in other ways:

\begin{theo}\label{th:1706a1}
Assume that, for any $y\in I$,
$|\eta|$ and $\frac1{|\eta|}$
are bounded on $(l,y)$
(the bounds may depend on~$y$)
and that $\sup\supp\mu<\infty$.
Then the processes $(Y^N_{Nt})_{t\in[0,1]}$
converge to $(M_t)_{t \in [0,1]}$ in distribution,
as $N \to \infty$.
\end{theo}

The proof is similar to that of
Theorem~\ref{convi sums2}.
Clearly, Theorem~\ref{th:1706a1}
has its analogue ``at~$r$''.

\subsection*{Examples}

We close the section by illustrating our results with several examples.
\begin{ex}[Brownian motion]
\label{ex:bm}
Let $M$ be a Brownian motion starting from some
$m\in\R$,
i.e.\ we have $l=-\infty$, $r=\infty$ and $\eta\equiv1$.
Then $q(y,x)=(x-y)^2$, for $y,x\in\R$, and
\begin{gather*}
G_y(a)=a^2\int x^2\,\mu(dx),\quad y\in\R,\;a\geq0.
\end{gather*}
Therefore, condition~(A1) of Section~\ref{subsec:2306a1}
is satisfied if and only if $\sigma^2:=\int x^2\,\mu(dx)<\infty$.
In this case, the scaled random walk $(Y^N_k)$ is determined by the scale factor
\be
a_N(y)=\frac1{\sqrt{N\sigma^2}},
\ee
which does not depend on the state~$y$. %Notice that $(Y^N_k)$ coincides with the random walk from the Donsker--Prokhorov invariance principle.
Thus, since (C1) is satisfied,
Theorem~\ref{convi sums}
yields weak convergence of $(Y^N_{N t })$ to $(M_t)$
under the assumptions that
$\mu\ne\delta_0$ is centered
and $\int x^2\,\mu(dx)<\infty$. This is exactly the Donsker--Prokhorov
invariance principle.
\end{ex}

\begin{ex}[Diffusion between two media]
\label{ex:2media}
Let $l=-\infty$, $r=\infty$ and, with some
$A\in\R\setminus\{0\}$,
\be
\eta(x)=1_{(0,\infty)}(x)+A1_{(-\infty,0]}(x),
\quad x\in\R.
\ee
Notice that we have
\be
\text{for } y \ge 0: \quad q(y,x) &=&
\begin{cases}
(x-y)^2, & x\geq0,\\
y^2 - 2xy + \frac{1}{A^2} x^2, & x<0,
\end{cases}\\[1mm]
\text{for } y \le 0: \quad q(y,x) &=&
\begin{cases}
\frac{1}{A^2} (x-y)^2, & x<0,\\
\frac{1}{A^2} y^2 - \frac{2}{A^2} xy + x^2, & x\geq0.
\end{cases}
\ee
Since, for appropriate $0<c_1<c_2<\infty$,
we have $c_1 (x-y)^2 \le q(y,x) \le c_2 (x-y)^2$,
condition~(A1) is satisfied if and only if $\mu$ has a finite second moment.
Again, (C1)~is satisfied, hence
the processes
$(Y^N_{Nt})$ converge in distribution to $(M_t)$
for any such~$\mu$.
\end{ex}

\begin{ex}[Geometric Brownian motion]
\label{ex:gbm}
Let $l=0$, $r=\infty$ and $\eta(x) = x$ on~$I$.
For $y>0$, we have
\be
q(y,x) =\begin{cases}
2\frac{x-y}{y} - 2 \log\frac xy, & \text{ if } x > 0, \\
\infty, & \text{ if } x \le 0.
\end{cases}
\ee
Since, for fixed $y>0$, $q(y,x)$ has linear growth
as $x\to\infty$, condition~(A2)
of Section~\ref{subsec:2306a2}
is satisfied
if and only if $\inf\supp\mu>-\infty$.
For all such measures~$\mu$,
\eqref{cond a3lrinfty}~is satisfied
due to Proposition~\ref{suffcond_2},
and hence Theorem~\ref{existence_case2} applies;
that is, for sufficiently large $N\in\N$,
Problem~(P) has a solution
with scale factor $a_N$
satisfying~\eqref{sf eq}.
Since (C2)~holds true,
by Theorem~\ref{convi sums2},
the processes
$(Y^N_{Nt})$ converge in distribution to $(M_t)$
for any~$\mu$ with a compact support.
\end{ex}

%==========
\appendix
\section{Appendix}
We use the setting and notations of Section~\ref{sec1}.
In particular, we consider a weak solution
$(M,W)$ of~\eqref{sde}, where the initial
condition $M_0$ has distribution $\gamma$,
and we treat the embedding problem~\eqref{eq:eiit7},
where $a\colon I\to(0,\infty)$ is a given Borel function.
Let us now briefly explain,
following~\cite{AHS}, a solution method
of~\eqref{eq:eiit7},
which gives an embedding stopping time
satisfying~\eqref{eq:100914a3}
provided~\eqref{eq:100914a2} holds true.

Let $\tilde W$ be an $(\tilde\cF_t)$-Brownian motion on some
$(\tilde\Omega,\tilde\cF,(\tilde\cF_t),\tilde P)$
and $\tilde M_0$ an $\tilde\cF_0$-measurable
random variable with distribution~$\gamma$.
For $y\in I$,
let $F_y$ and $F_\mu$
be the distribution functions of $K(y,a(y),\cdot)$
and of $\mu$,
as well as $F_y^{-1}$ and $F_\mu^{-1}$
their generalized inverse functions
(that is,
$F_y^{-1}(r)=\inf\{x\in\R:F_y(x)>r\}$,
$r\in(0,1)$,
and the same formula holds for~$F_\mu^{-1}$).
For $y\in I$, $t\in[0,1]$ and $x\in\R$, we define
\be
g(y,t,x)&=&\tilde E
[F_y^{-1} \circ \Phi(\tilde W_1) | \tilde W_t=x],\\
b(t,x)&=&\tilde E
[F_\mu^{-1} \circ \Phi(\tilde W_1) | \tilde W_t=x],
\ee
where $\Phi$ denotes the
standard normal distribution function,
and notice that
\ben\label{eq:100914a4}
g(y,t,x) = y + a(y) b(t,x).
\een
Let us define the $(\tilde\cF_t)$-martingale
$N_t=b(t,\tilde W_t)$, $t\in[0,1]$,
and the process
$L_t=g(\tilde M_0,t,\tilde W_t)
=\tilde M_0+a(\tilde M_0)N_t$,
$t\in[0,1]$
(the latter process can fail to be a martingale
because it can fail to be integrable).
Observe that
$N_1$ has the distribution $\mu$,
hence
\ben\label{eq:ahsc0}
\Law(L_1 | \tilde\cF_0)=K(\tilde M_0,a(\tilde M_0),\cdot).
\een
Moreover, we have
\ben\label{eq:ahsc1}
\tilde P\big((L_t)_{t\in[0,1]}\in A | \tilde\cF_0\big)
=G(\tilde M_0,A),\quad A\in\cB(C[0,1]),
\een
where the kernel $G$ is given by the formula
\ben\label{eq:ahsc2}
G(y,A)=\tilde P\big(
\big(g(y,t,\tilde W_t)\big)_{t\in[0,1]}\in A\big),
\quad y\in I,\;A\in\cB(C[0,1]).
\een
One can also check that
the function $b$ is smooth on $[0,1)\times\R$ and,
for any $t\in[0,1)$,
the function $b(t,\cdot)$
is a strictly increasing
bijective mapping
$\R\to(\inf\supp\mu,\sup\supp\mu)$.
Let $g^{-1}$ denote the inverse of $g$
in the last argument, which is well defined
when the se\-cond argument $t\in[0,1)$.

A straightforward generalization
of Theorems~1 and~3
and Lemma~2 in~\cite{AHS}
now yields the following statement.

\begin{propo}\label{prop:100914a1}
Assume that \eqref{eq:100914a2} holds true.
Then the ODE
\ben\label{ode}
\delta'(t) = \frac{a^2(M_0)
b_x^2(t,g^{-1}(M_0,t,M_{\delta(t)}))}{\eta^2(M_{\delta(t)})},
\quad t\in[0,1),
\quad \delta(0) = 0,
\een
has a solution on $[0,1)$ for $P$-almost all paths.
Here, $b_x$ denotes the partial derivative
of $b$ with respect to the second argument.
We set
\ben\label{odea}
\delta(1)=\lim_{t\uparrow1}\delta(t),
\een
which is well defined $P$-a.s.\
because $\delta$ is nondecreasing.
Moreover,
$(\delta(t))_{t\in[0,1]}$
is an $(\cF_t)$-time change,
the $(\cF_t)$-stopping time $\delta(1)$ satisfies
\ben\label{eq:100914a5}
E[\delta(1)|\cF_0]=Q(M_0)
\quad P\text{-a.s.},
\een
the process
\ben\label{eq:100914a10}
Z_t=\frac1{a(M_0)}(M_{\delta(t)}-M_0),
\quad t\in[0,1],
\een
is an $(\cF_{\delta(t)})$-martingale,
and
\ben\label{eq:100914a11}
\Law(Z_t;\,t\in[0,1]\,|\,\cF_0)
=\Law(N_t;\,t\in[0,1])
\quad P\text{-a.s.},
\een
where the left-hand side
is the notation for
the regular conditional distribution
of the process $(Z_t)_{t\in[0,1]}$
with respect to~$\cF_0$,
while the right-hand side
is the notation for
the unconditional distribution
of the process $(N_t)_{t\in[0,1]}$
(that is, the former, which is in general
a kernel depending on $\omega$,
equals the latter for almost all paths).
\end{propo}

\begin{corollary}\label{cor:100914a1}
Assume that \eqref{eq:100914a2} holds true.
Then
\ben\label{eq:100914a12}
\Law(M_{\delta(t)};\,t\in[0,1]\,|\,\cF_0)
=G(M_0,\cdot)
\quad P\text{-a.s.},
\een
where the kernel $G$ is given by~\eqref{eq:ahsc2}
(recall the relation between the processes
$(N_t)$ and $(L_t)$
right after~\eqref{eq:100914a4}).
In particular, $\delta(1)$ is a solution
of the embedding problem~\eqref{eq:eiit7}
satisfying~\eqref{eq:100914a5}
(see~\eqref{eq:ahsc0} and~\eqref{eq:ahsc1}).
\end{corollary}

The next lemma summarizes
the properties we need in this paper.

\begin{lemma}\label{lemma st properties}
Assume~\eqref{eq:100914a2}.
Then the following holds true.

\smallskip
(i) The process
\be
X_t=\frac1{a(M_0)}(M_{\delta(1)\wedge t}-M_0),
\quad t\geq0,
\ee
is a uniformly integrable
$(\cF_t)$-martingale.

\smallskip
(ii) The $(\cF_t)$-stopping time $\delta(1)$
has the same distribution as the random variable
\ben\label{eq:140914a1}
\xi=
\int_0^1 \frac{a^2(\tilde M_0) b^2_x(s,\tilde W_s)}{\eta^2(\tilde M_0 + a(\tilde M_0) b(s,\tilde W_s))}\,ds.
\een
(Of course one can drop the tildes
in the latter formula.)
\end{lemma}

\begin{proof}
(i) First observe that the $(\cF_t)$-time change
$(\delta(t))_{t\in[0,1]}$
is $P$-a.s.\ strictly increasing on $[0,1]$.
Indeed, if it had an interval of constancy,
then, by~\eqref{eq:100914a10}
and~\eqref{eq:100914a11},
the process $(N_t)_{t\in[0,1]}$
would have an interval of constancy,
which is impossible because
$N_t=b(t,\tilde W_t)$ and,
for $t\in[0,1)$,
$b$ is smooth in both arguments
and $b(t,\cdot)$ is strictly increasing.
Thus, the inverse $\delta^{-1}$
is well defined.

Now, for a fixed $t\geq0$,
define $\eta=\delta^{-1}(\delta(1)\wedge t)$.
Since $\delta(1)\wedge t$ is an $(\cF_t)$-stopping time,
$\eta$ is an $(\cF_{\delta(t)})$-stopping time.
Clearly, $\eta\leq1$.
Doob's optional sampling theorem
applied to the $(\cF_{\delta(t)})$-martingale
$(Z_t)_{t\in[0,1]}$ (see~\eqref{eq:100914a10})
and to the bounded $(\cF_{\delta(t)})$-stopping times
$\eta$ and~$1$
yields
$E(Z_1|\cF_{\delta(\eta)})=Z_\eta$, $P$-a.s.,
which is equivalent to
\be
E(X_\infty|\cF_{\delta(1)\wedge t})=X_t
\quad P\text{-a.s.}
\ee
A short calculation reveals that,
since the process $(X_t)_{t\geq0}$
is stopped at $\delta(1)$,
we also have
\be
E(X_\infty|\cF_t)=X_t
\quad P\text{-a.s.}
\ee
This concludes the proof of~(i).

\smallskip
(ii) Formulas \eqref{ode},
\eqref{eq:100914a12},
\eqref{eq:ahsc2},~\eqref{eq:ahsc1}
as well as
\be
g^{-1}(\tilde M_0,t,L_t)=\tilde W_t
\quad\text{and}\quad
L_t=\tilde M_0+a(\tilde M_0)b(t,\tilde W_t)
\ee
immediately imply
\ben\label{eq:140914a2}
\Law(\delta(1)|\cF_0)=H(M_0,\cdot)
\quad P\text{-a.s.},
\een
where the kernel $H$ is given by the formula
\be
H(y,\cdot)=\Law\left(
\int_0^1 \frac{a^2(y)b_x^2(s,\tilde W_s)}{\eta^2(y+a(y)b(s,\tilde W_s))}\,ds\right),
\quad y\in I.
\ee
Since $\tilde M_0$ is $\tilde\cF_0$-measurable
and the process $(\tilde W_s)$
is independent of $\tilde\cF_0$,
then for the random variable
$\xi$ of~\eqref{eq:140914a1} we get
\ben\label{eq:140914a3}
\Law(\xi|\tilde\cF_0)=H(\tilde M_0,\cdot)
\quad\tilde P\text{-a.s.}
\een
The statement now follows from
\eqref{eq:140914a2},~\eqref{eq:140914a3}
and the fact that $M_0$
and $\tilde M_0$
have the same distribution.
\end{proof}

Sometimes we use the notation $\Delta = \delta(1)$
and also write $\Delta(M_0)$ instead of $\Delta$
whenever we want to stress
the dependence on~$M_0$.

\bibliographystyle{plain}
\bibliography{literature}

\end{document}